\documentclass[11pt]{amsart}

\usepackage{amsmath,amssymb,amsthm,mathtools}
\usepackage{booktabs}
\usepackage{longtable}
\usepackage{enumitem}

\newtheorem{theorem}{Theorem}[section]
\newtheorem{proposition}[theorem]{Proposition}
\newtheorem{lemma}[theorem]{Lemma}
\newtheorem{corollary}[theorem]{Corollary}
\newtheorem{remark}[theorem]{Remark}

\newcommand{\Aut}{\operatorname{Aut}}
\newcommand{\Pone}{\mathbb P^1}
\newcommand{\Char}{\operatorname{char}}
\newcommand{\lcm}{\operatorname{lcm}}
\newcommand{\wpw}{\wp_2}

\title{Cyclic automorphisms beyond the canonical bound}

\author{Ahmad Kazemifard}
\address{Department of Mathematics, Faculty of Mathematical Sciences and Computer,
Shahid Chamran University of Ahvaz, Ahvaz, Iran}
\email{a.kazemifard@scu.ac.ir}

\author{Saeed Tafazolian}
\address{Department of Mathematics, IMECC, University of Campinas (UNICAMP), Campinas, SP, Brazil}
\email{saeed@unicamp.br}
\date{}

\begin{document}

\begin{abstract}
Let $X$ be a smooth projective curve of genus $g\geq2$ over an
algebraically closed field $K$, and suppose that $\Aut(X)$ contains a
cyclic subgroup $G$ of order $N>2g-2$.  The cases $N\geq2g+1$ are known.
We treat the two boundary values $N=2g$ and $N=2g-1$.  In the tame case
we obtain an explicit finite list of ramification signatures.  In odd
characteristic we prove that the Sylow $p$-subgroup of $G$ has order
$p$ and determine all wild boundary cases by a direct ramification
analysis; the resulting curves are Artin--Schreier--Kummer fiber
products.  In characteristic two we prove that the $2$-part has order at most four
and classify all wild boundary cases, including Artin--Schreier--Witt
normal forms for the $C_4$ and $C_{12}$ cases.
\end{abstract}

\maketitle

\section{Introduction}

Let $K$ be an algebraically closed field and let $X/K$ be a smooth
projective geometrically irreducible curve of genus $g\geq2$.  We study
cyclic subgroups
\[
        G=\langle\sigma\rangle\cong C_N\leq\Aut(X)
\]
whose order lies just above the degree $2g-2$ of the canonical divisor.

Over the complex numbers, Harvey gave necessary and sufficient conditions
for a Fuchsian signature to admit a surface-kernel epimorphism onto a
cyclic group \cite{Harvey1966}.  Nakagawa studied upper bounds and large
orders for cyclic automorphisms of compact Riemann surfaces
\cite{Nakagawa1984}.  Irokawa and Sasaki classified the range
$N\geq2g+1$ in characteristic zero \cite{IrokawaSasaki1995}.

In positive characteristic, the odd-characteristic analogue of the
Irokawa--Sasaki theorem was proved by Dionigi, Giulietti and Timpanella
\cite{DGT2026}.  The characteristic-two case was proved by Timpanella
\cite{Timpanella2026}.  Thus the only integer values strictly larger than
$2g-2$ not covered by those theorems are
\[
        N=2g,\qquad N=2g-1.
\]
The purpose of this paper is to treat these two values.

Several results on large prime-order and prime-square automorphisms are
closely related but do not subsume this boundary problem.  Arakelian and
Speziali classify curves admitting prime-order automorphisms in the cases
$q=2g+1$ and $q=g+1$, and obtain partial results for $q=g,g-1$
\cite{ArakelianSpeziali2021}.  In the form quoted in \cite[Theorem~2.3]{ArakelianSpeziali2021},
Homma's prime-order theorem \cite{Homma1980} says that a prime order
$q$ satisfies $q\leq g+1$, or $q=2g+1$, with the case
$q=2g-1$ occurring only when $g=2$.  Thus the prime-order overlap with
the present $N=2g-1$ boundary is only the genus-two case $q=3$.  Their earlier work on generalized Fermat curves concerns
a different problem: direct products $C_m\times C_n$ of tame cyclic
groups over finite fields, with both quotient curves rational
\cite{ArakelianSpeziali2017}.  More recently, Arakelian and Kian studied
automorphism groups of prime-square order \cite{ArakelianKian2026}; that
work is relevant background for the special $\ell^2$ subproblem, but not
a classification of cyclic groups of arbitrary composite order at
$N=2g$ or $N=2g-1$.

The tame signatures below are consequences of the classical cyclic-cover
conditions, and no claim of novelty is made for their existence over
$\mathbb C$.  The main point is the uniform positive-characteristic
boundary classification, especially the wild cases.

In characteristic two, the argument is modeled on Timpanella's treatment
of the range $N\geq2g+1$ \cite[Section~3]{Timpanella2026}: the boundary
cases whose $2$-part has order two arise as equality cases of the
inequalities established there.  The cases with $4\mid N$, namely the
$C_4$-family in genus two and the $C_{12}$-family in genus six
(Theorems~\ref{thm:C4normal} and~\ref{thm:C12normal}), have no
counterpart in the range $N\geq2g+1$, where $4\nmid N$
\cite[Theorem~1.1]{Timpanella2026}.

\section{Preliminaries}

For a finite subgroup $H\leq\Aut(X)$ we write $X/H$ for the nonsingular
model of $K(X)^H$.  Throughout the paper, an affine Kummer,
Artin--Schreier, or Artin--Schreier--Witt equation denotes the smooth
projective normalization of the corresponding affine curve.  If $P\in X$,
let $H_P$ be the stabilizer of $P$ and
$H_P^{(i)}$ the $i$th lower ramification group.  When upper numbering is
used, it is related to the lower numbering by Herbrand's function; we use
the conventions of \cite[Chapter~IV]{Serre1979}.  Put
\[
       d_P=\sum_{i\geq0}\bigl(|H_P^{(i)}|-1\bigr).
\]
If $Q\in X/H$ is a branch point and $P$ lies above $Q$, set
\[
       e_Q=|H_P|,\qquad \delta_Q=\frac{d_P}{e_Q}.
\]
The Riemann--Hurwitz formula divided by $|H|$ is
\[
 \frac{2g(X)-2}{|H|}
 =
 2g(X/H)-2+\sum_Q\delta_Q.
 \tag{2.1}
\]

\subsection{The quotient}

\begin{proposition}\label{prop:rationalquotient}
If $G\cong C_N$ and $N>2g-2$, then
\[
        X/G\cong\Pone,
\]
and
\[
        \sum_Q\delta_Q
        =
        2+\frac{2g-2}{N}
        <3.
        \tag{2.2}
\]
\end{proposition}

\begin{proof}
Let $g_0=g(X/G)$.  Since $0<(2g-2)/N<1$, the cover is not unramified.
If $g_0\geq2$, then the right-hand side of \textup{(2.1)} is at least
two, a contradiction.

Suppose $g_0=1$.  A wild branch point satisfies $\delta_Q\geq1$, already
contradicting \textup{(2.1)}.  Hence every ramified point is tame.  A nontrivial cyclic tame cover cannot ramify at exactly one point.
Indeed, by Kummer theory one may write the relevant cyclic quotient as
$K(X)=K(X/G)(f^{1/d})$.  If all valuations of $f$ away from a single
point $Q$ are divisible by $d$, then the degree-zero condition on
$\operatorname{div}(f)$ forces $v_Q(f)$ to be divisible by $d$ as well,
so $Q$ is unramified.  Thus there are at least two branch points, and
each contributes at least $1/2$.  Therefore
\[
        \frac{2g-2}{N}=\sum_Q\delta_Q\geq1,
\]
again a contradiction.  Hence $g_0=0$, and \textup{(2.2)} follows from
\textup{(2.1)}.
\end{proof}

In characteristic two, the same statement, with the same proof, appears in
\cite[Lemma~2.6]{Timpanella2026}.

\subsection{Cyclic tame covers of the projective line}

\begin{lemma}\label{lem:lcm}
Let $Y\to\Pone$ be a nontrivial cyclic tame cover of degree $m$, with
ramification indices $e_1,\ldots,e_r$.  Then
\[
  \lcm(e_1,\ldots,e_r)=m,
\]
and, for every $i$,
\[
  \lcm(e_1,\ldots,\widehat{e_i},\ldots,e_r)=m.
  \tag{2.3}
\]
In particular, if $r=2$, then $e_1=e_2=m$ and $Y$ is rational.  If
$r=3$ and $d_i=m/e_i$, then $d_1,d_2,d_3$ are pairwise coprime.
\end{lemma}

\begin{proof}
Fix $i$ and let $B_i$ be the subgroup generated by the stabilizers over
all branch points except the $i$th one.  If $B_i\neq C_m$, then
$Y/B_i\to\Pone$ is a nontrivial tame cyclic cover ramified over at most
one point.  This is impossible: if its degree is $d>1$, Riemann--Hurwitz
would give
\[
  2g(Y/B_i)-2\leq -2d+(d-1)<-2.
\]
Thus $B_i=C_m$.  In a cyclic group, subgroups generate the full group if
and only if the least common multiple of their orders is the group order,
which proves \textup{(2.3)}.

If $r=2$, omitting either index gives $e_1=e_2=m$; Riemann--Hurwitz
then gives $g(Y)=0$.  Finally, if $r=3$ and a prime divided both $d_i$
and $d_j$, then $\lcm(e_i,e_j)<m$, contradicting \textup{(2.3)} after
omitting the third index.
\end{proof}

Compare \cite[Lemma~2.1]{Timpanella2026}, where the same statements are
proved for Kummer covers of odd degree.

\subsection{A local lemma for inertia with $p$-part of order $p$}

\begin{lemma}\label{lem:oddlocal}
Assume $\Char(K)=p>0$.  Let a cyclic inertia group have order $pt$, where
$(p,t)=1$.  Then there is a positive integer $j$ such that
\[
 C_p=I^{(1)}=\cdots=I^{(j)},\qquad I^{(j+1)}=1,
\]
and
\[
        t\mid j.
\]
Writing $j=tn$, one has $p\nmid n$, and
\[
 \delta
 =
 1-\frac1{pt}+\frac{n(p-1)}p.
 \tag{2.4}
\]
Moreover, after quotienting by the tame subgroup $C_t$, the resulting
Artin--Schreier extension has reduced pole order $n$.
\end{lemma}

\begin{proof}
Since the wild Sylow subgroup has order $p$, all positive ramification
groups are either $C_p$ or trivial.  Hence they have the displayed form.

Choose a local parameter $v$ at the ramified point so that a generator
$\tau$ of the tame subgroup $C_t$ satisfies
\[
        \tau(v)=\zeta_t v.
\]
If $\rho$ generates $C_p$, write
\[
        \rho(v)=v+c v^{j+1}+\text{higher terms},
        \qquad c\neq0.
\]
The inertia group is cyclic, so $\rho\tau=\tau\rho$.  Comparing the
coefficients of $v^{j+1}$ gives
\[
        \zeta_t c=\zeta_t^{j+1}c,
\]
and hence $t\mid j$.

The different exponent is
\[
        d=(pt-1)+j(p-1).
\]
Writing $j=tn$ and dividing by $pt$ gives \textup{(2.4)}.  The quotient
by $C_t$ is a degree-$p$ Artin--Schreier extension whose lower break is
$j/t=n$.  In reduced Artin--Schreier form, the lower break is the pole
order, and therefore $p\nmid n$; see, for example,
\cite[Proposition~3.7.8]{Stichtenoth2009}.
\end{proof}

For $p=2$, Lemma~\ref{lem:oddlocal} is \cite[Lemma~2.10(1)]{Timpanella2026}.

\subsection{Excluding a cyclic $p^2$ inertia group}

\begin{lemma}\label{lem:p2local}
Let $p\geq3$ and let $L/K$ be a totally ramified cyclic extension of
complete discretely valued fields of characteristic $p$ with
$\operatorname{Gal}(L/K)\cong C_{p^2}$.  Then
\[
 \frac{d(L/K)}{p^2}
 \geq
 p+\frac1p-\frac{2}{p^2}
 >3.
 \tag{2.5}
\]
\end{lemma}

\begin{proof}
Let $u_1<u_2$ be the upper jumps.  Artin--Schreier--Witt ramification
theory gives
\[
        u_1\geq1,\qquad u_2\geq p u_1;
\]
see, for example, \cite{ElderKeating2026,Obus2012}.  The lower jumps are
\[
        b_1=u_1,\qquad b_2=u_1+p(u_2-u_1).
\]
Hilbert's different formula gives
\[
 d(L/K)
 =
 (b_1+1)(p^2-1)+(b_2-b_1)(p-1).
\]
The minimum occurs for $u_1=1$ and $u_2=p$, and yields
\[
        d(L/K)\geq p^3+p-2.
\]
Dividing by $p^2$ proves \textup{(2.5)}.
\end{proof}

\begin{corollary}\label{cor:p2inertia}
If a cyclic inertia group in odd characteristic contains $C_{p^2}$, then
its normalized different contribution is greater than $3$.
\end{corollary}

\begin{proof}
Let $L/K$ be the corresponding totally ramified local extension and let
$I=\operatorname{Gal}(L/K)$ have order $p^a t$, where $a\geq2$ and
$(p,t)=1$.  Let $H\leq I$ be the unique subgroup of order
$p^{a-2}t$, and let $E=L^H$.  Then
$\operatorname{Gal}(E/K)\cong C_{p^2}$.  By transitivity of the different,
\[
 d(L/K)=d(L/E)+|H|\,d(E/K)\geq |H|\,d(E/K).
\]
Since $|I|=|H|p^2$, Lemma~\ref{lem:p2local} gives
\[
        \frac{d(L/K)}{|I|}
        \geq
        \frac{d(E/K)}{p^2}>3.
\]
\end{proof}

\begin{lemma}[Generation by point stabilizers]\label{lem:generate-stab}
Let a finite abelian group $A$ act on a smooth projective curve $V$, and
assume that $V/A\cong\Pone$.  Then $A$ is generated by its point
stabilizers.  In particular, if $A$ is cyclic, then for every prime
$\ell\mid |A|$ some point stabilizer contains the full Sylow
$\ell$-subgroup of $A$.
\end{lemma}

\begin{proof}
Let $B\leq A$ be generated by all point stabilizers.  Every inertia group
maps trivially to $A/B$, so the induced cover
\[
        V/B\longrightarrow V/A\cong\Pone
\]
is unramified.  If $B\neq A$, Riemann--Hurwitz would give
\[
        2g(V/B)-2=-2[A:B]<-2,
\]
which is impossible.  Hence $B=A$.  If $A$ is cyclic, its
$\ell$-subgroups form a chain, so the subgroup generated by the point
stabilizers can contain the full Sylow $\ell$-subgroup only if one of
the stabilizers already contains it.
\end{proof}

In characteristic two, the same argument appears in
\cite[Lemma~2.6]{Timpanella2026}.

\subsection{A fiber-product lemma}

\begin{lemma}\label{lem:fiberproduct}
Assume $p\mid |G|$, $G=C_p\times C_m$, and $(p,m)=1$.  Put
\[
        S=C_p,\qquad M=C_m,\qquad Z=X/S,\qquad W=X/G.
\]
Then the Artin--Schreier cover $X\to Z$ is the pullback of an
Artin--Schreier cover of $W$.
\end{lemma}

\begin{proof}
Choose an Artin--Schreier generator $y$ for $K(X)/K(Z)$ so that a
generator of $S$ sends $y$ to $y+1$.  For $\tau\in M$, commutativity of
$S$ and $M$ implies
\[
        \tau(y)-y\in K(Z).
\]
The map $\tau\mapsto\tau(y)-y$ is a $1$-cocycle of $M$ with values in
the additive $M$-module $K(Z)$.  Since $m$ is invertible in $K$,
\[
        H^1(M,K(Z))=0.
\]
After replacing $y$ by $y-h$ for a suitable $h\in K(Z)$, we may therefore
assume that $M$ fixes $y$.  Hence
\[
        y^p-y=f\in K(Z)^M=K(W),
\]
and therefore
\[
        K(X)=K(Z)(y),\qquad y^p-y=f\in K(W).
\]
Equivalently, $X$ is the normalization of the fiber product of
$Z\to W$ with the Artin--Schreier cover of $W$ defined by $y^p-y=f$.
This argument also applies when $p=2$.
\end{proof}

\section{The tame boundary}

Assume $p\nmid N$.

\begin{theorem}[Tame boundary classification]\label{thm:tame}
Assume that the action of $G\cong C_N\leq\Aut(X)$ is tame; equivalently,
either $\Char(K)=0$, or $\Char(K)=p>0$ and $p\nmid N$.

\begin{enumerate}[label=\textup{(\roman*)}]
\item If $N=2g$, then $X/G$ has four branch points and, up to permutation,
the ramification signature is one of
\[
\begin{split}
 &(2,2,N,N),\\
 &(2,3,3,6)\qquad (N=6),\\
 &(2,3,4,12)\qquad (N=12),\\
 &(3,3,4,4)\qquad (N=12),\\
 &(2,3,5,30)\qquad (N=30).
\end{split}
\]

\item If $N=2g-1$, then $X/G$ has four branch points, and
$(N;\text{signature})$ is either
\[
        (3;\,3,3,3,3)\qquad\text{or}\qquad(15;\,3,3,5,5).
\]
\end{enumerate}
\end{theorem}

\begin{proof}
By Proposition~\ref{prop:rationalquotient}, $X/G$ is rational.  Let
$e_1\leq\cdots\leq e_r$ be the branch indices.

Assume first that $N=2g$.  Riemann--Hurwitz gives
\[
        \sum_{i=1}^{r}\frac1{e_i}
        =
        r-3+\frac2N.
        \tag{3.1}
\]
The case $r=3$ is impossible, since the left-hand side is at least
$3/N$ while the right-hand side is $2/N$.  Since $e_i\geq2$,
\[
        r-3+\frac2N\leq\frac r2,
\]
so $r\leq5$.

Suppose $r=5$.  Equation \textup{(3.1)} becomes
\[
        \sum_{i=1}^{5}\frac1{e_i}=2+\frac2N.
\]
If $e_3\geq3$, the left-hand side is at most
$1/2+1/2+1/3+1/3+1/3=2$, a contradiction.  Thus
$e_1=e_2=e_3=2$, and
\[
        \frac1{e_4}+\frac1{e_5}=\frac12+\frac2N.
\]
If $e_4\geq4$, the left-hand side is at most $1/2$, impossible.  If
$e_4=2$, then $e_5=N/2$, but Lemma~\ref{lem:lcm}, after omitting
$e_5$, gives $N=2$, contradicting $g\geq2$.  If $e_4=3$, omitting $e_5$ gives $N=6$, and the
last equation gives $e_5=2$, contrary to the ordering.  Hence $r=4$.

We now have
\[
 \frac1{e_1}+\frac1{e_2}+\frac1{e_3}+\frac1{e_4}
 =
 1+\frac2N.
 \tag{3.2}
\]
Thus $e_1\leq3$.

Suppose $e_1=2$.  If $e_2=2$, then
\[
        \frac1{e_3}+\frac1{e_4}=\frac2N.
\]
Since $e_3,e_4\leq N$, equality forces $e_3=e_4=N$.

Assume $e_2\geq3$.  From \textup{(3.2)} one gets $e_2\leq5$.

If $e_2=5$, then $e_3\leq6$.  For $e_3=5,6$, Lemma~\ref{lem:lcm}
gives respectively $N=10,30$, and \textup{(3.2)} gives
$e_4=10/3,5$, neither of which is admissible.

If $e_2=4$, then $e_3\leq7$.  For $e_3=4,5,6,7$, the equality
$N=\lcm(2,4,e_3)$ and \textup{(3.2)} give
\[
        e_4=2,\quad \frac{20}{3},\quad4,\quad\frac{28}{5},
\]
respectively.  None is compatible with the ordering.

Let $e_2=3$.  Then $e_3\leq11$ and, after omitting $e_4$,
\[
        N=\lcm(6,e_3).
\]
For $e_3=3,4,5$, equation \textup{(3.2)} gives respectively
\[
 (N,e_4)=(6,6),\qquad(12,12),\qquad(30,30).
\]
For $e_3=6,9$ one obtains $e_4<e_3$; for $e_3=10,11$ one obtains a
nonintegral value of $e_4$.  For $e_3=7$ one gets $(N,e_4)=(42,14)$,
but
\[
        \lcm(2,7,14)=14<42,
\]
contrary to Lemma~\ref{lem:lcm} after omitting $e_2$.  For $e_3=8$ one
gets $(N,e_4)=(24,8)$, but
\[
        \lcm(2,8,8)=8<24.
\]
This gives the three exceptional signatures beginning with $(2,3)$.

Suppose now that $e_1=3$.  If $e_2\geq5$, the left-hand side of
\textup{(3.2)} is at most $1/3+3/5<1$, impossible.  If $e_2=4$, then
\textup{(3.2)} forces $e_3=4$; Lemma~\ref{lem:lcm} gives $N=12$ and
then $e_4=3$, contrary to the ordering.  Hence $e_2=3$.  Then
$e_3\leq5$.  If $e_3=3$, Lemma~\ref{lem:lcm} gives $N=3$, impossible
because $N=2g$ is even.  If $e_3=4$, then $N=12$ and $e_4=4$.  If
$e_3=5$, then $N=15$, again impossible.  This proves (i).

Assume now that $N=2g-1$.  Then $N$ and all the $e_i$ are odd, so
$e_i\geq3$, and
\[
        \sum_{i=1}^{r}\frac1{e_i}
        =
        r-3+\frac1N.
        \tag{3.3}
\]
The case $r=3$ is impossible because $3/N>1/N$, while
$\sum 1/e_i\leq r/3$ shows that $r\leq4$.  Hence $r=4$ and
\[
 \frac1{e_1}+\frac1{e_2}+\frac1{e_3}+\frac1{e_4}
 =
 1+\frac1N.
 \tag{3.4}
\]
If $e_1\geq5$, the left-hand side is at most $4/5<1$, so $e_1=3$.
If $e_2\geq5$, it is at most $1/3+3/5<1$, so $e_2=3$.  Consequently
\[
        \frac1{e_3}+\frac1{e_4}=\frac13+\frac1N.
\]
If $e_3\geq7$, the left-hand side is at most $2/7<1/3$, so
$e_3\in\{3,5\}$.  For $e_3=3$, Lemma~\ref{lem:lcm} gives $N=3$ and
then $e_4=3$.  For $e_3=5$, it gives $N=15$ and then $e_4=5$.
This proves (ii).
\end{proof}

\subsection{Kummer representatives}

For a tame cyclic cover of degree $N$ with branch points
$0,1,\lambda,\infty$, write $a_0,a_1,a_\lambda,a_\infty\in
\mathbb Z/N\mathbb Z$ for the Kummer exponent residues.  They satisfy
\[
 a_0+a_1+a_\lambda+a_\infty\equiv0\pmod N,\qquad
 e_i=\frac{N}{\gcd(N,a_i)}.
\]
Changing the generator of $C_N$ multiplies all four residues by a unit
modulo $N$, while permuting branch points permutes the residues.

With $\lambda\in K\setminus\{0,1\}$, the following models realize all
the signatures in Theorem~\ref{thm:tame}:
\[
\begin{array}{c|c|c}
\text{signature} & g & \text{Kummer model}\\
\hline
(2,2,N,N) & N/2 &
 y^N=x(x-1)^{N/2}(x-\lambda)^{N/2}\\[2mm]
(2,3,3,6) & 3 &
 y^6=x^3(x-1)^2(x-\lambda)^2\\[2mm]
(2,3,4,12) & 6 &
 y^{12}=x^6(x-1)^4(x-\lambda)^3\\[2mm]
(3,3,4,4) & 6 &
 y^{12}=x^4(x-1)^8(x-\lambda)^3\\[2mm]
(2,3,5,30) & 15 &
 y^{30}=x^{15}(x-1)^{10}(x-\lambda)^6\\[2mm]
(3,3,3,3) & 2 &
 y^3=x(x-1)(x-\lambda)^2\\[2mm]
(3,3,5,5) & 8 &
 y^{15}=x^5(x-1)^{10}(x-\lambda)^3 .
\end{array}
\]

For completeness, we verify the exponent data.  For the family
$(2,2,N,N)$, the two residues with inertia two are both $N/2$, while
the two residues with full inertia are $u,-u$ with
$u\in(\mathbb Z/N\mathbb Z)^\times$; multiplication by $u^{-1}$ gives
the displayed model.

For the six sporadic cases the check is finite and can be made directly.
We record it, since the inertia indices alone do not determine Kummer
residues in general.

For $(2,3,3,6)$ the residue of inertia two is $3$, the residues of
inertia three are in $\{2,4\}$, and the full-inertia residue is in
$\{1,5\}$.  The zero-sum condition forces either
$(3,2,2,5)$ or $(3,4,4,1)$, and the latter is obtained from the former
by multiplication by $5$.

For $(2,3,4,12)$ the inertia-two residue is $6$.  Multiplying by a unit,
we may assume that the inertia-three residue is $4$.  The inertia-four
residue is then $3$ or $9$.  The zero-sum condition gives respectively
$(6,4,3,11)$ and $(6,4,9,5)$; multiplication of the second vector by
$7$ gives the first.

For $(3,3,4,4)$ the two inertia-three residues lie in $\{4,8\}$ and the
two inertia-four residues lie in $\{3,9\}$.  If the first two residues
were equal, their sum could not be cancelled by the sum of two elements
of $\{3,9\}$.  Hence they are $4,8$, and then the other two must be
$3,9$.  Thus the vector is $(4,8,3,9)$ up to permutation.

For $(2,3,5,30)$ the inertia-two residue is $15$.  We may first
normalize the inertia-three residue to $10$.  The units
$1,7,13,19$ preserve this residue and act transitively on
$\{6,12,18,24\}$, the possible inertia-five residues.  We may therefore
normalize the latter to $6$, after which the zero-sum condition forces
the full-inertia residue to be $29$.  This gives $(15,10,6,29)$.

For $(3,3,3,3)$ all four residues lie in $\{1,2\}$.  If $k$ of them are
$1$, their sum is congruent to $2-k$ modulo $3$, so $k=2$.  This gives
$(1,1,2,2)$ up to permutation.

Finally, for $(3,3,5,5)$ the two inertia-three residues lie in
$\{5,10\}$ and the two inertia-five residues lie in
$\{3,6,9,12\}$.  The sum of the inertia-three residues must be divisible
by $3$, so it is $0$ modulo $15$; hence these residues are $5,10$.
The other two must then also sum to zero and are either $3,12$ or
$6,9$.  Multiplication by the unit $2$ interchanges these two
possibilities.  Thus we may take $(5,10,3,12)$.

Consequently, up to multiplication by a unit and permutations among
points with the same inertia index, the six sporadic exponent vectors
are
\[
\begin{array}{c|c}
(2,3,3,6)&(3,2,2,5)\\
(2,3,4,12)&(6,4,3,11)\\
(3,3,4,4)&(4,8,3,9)\\
(2,3,5,30)&(15,10,6,29)\\
(3,3,3,3)&(1,1,2,2)\\
(3,3,5,5)&(5,10,3,12).
\end{array}
\]
These are precisely the exponent vectors of the displayed equations.
Thus every tame pair $(X,C_N)$ with one of these signatures is
$C_N$-equivariantly isomorphic to a member of the corresponding
one-parameter family.  The parameter $\lambda$ is a cross-ratio
parameter and is still subject to the finite transformations induced by
permutations of branch points; no uniqueness of $\lambda$ is asserted.

\section{Odd characteristic: the wild boundary}

Assume throughout this section that $p\geq3$ and $p\mid N$.

\subsection{The $p$-part}

\begin{proposition}\label{prop:p-part}
Write
\[
        N=p^\nu m,\qquad (p,m)=1.
\]
Then $\nu=1$.
\end{proposition}

\begin{proof}
By Proposition~\ref{prop:rationalquotient}, $X/G\cong\Pone$.
Lemma~\ref{lem:generate-stab} therefore shows that some point stabilizer
contains the full Sylow $p$-subgroup $C_{p^\nu}$ of $G$.  If $\nu\geq2$,
that stabilizer contains a subgroup $C_{p^2}$, so
Corollary~\ref{cor:p2inertia} gives a normalized different contribution
greater than $3$.  This contradicts
\[
        \sum_Q\delta_Q<3
\]
from \textup{(2.2)}.  Hence $\nu=1$.
\end{proof}

Put
\[
        S=C_p,\qquad M=C_m.
\]

\subsection{The master equation}

Let $q$ be the number of wild branch points of $X\to X/G$, and let
$s$ be the number of tame branch points.  At the $k$th wild branch point
write the inertia order as $pt_k$ and put
\[
        \bar d_k=\frac{m}{t_k},\qquad j_k=t_k n_k,
\]
where $p\nmid n_k$.  At the tame branch points let
\[
        e_i\mid m,\qquad d_i=\frac{m}{e_i}.
\]
Let
\[
        c=
        \begin{cases}
        2,&N=2g,\\
        1,&N=2g-1.
        \end{cases}
\]
Then Riemann--Hurwitz and Lemma~\ref{lem:oddlocal} give
\[
 m\left((q+s-3)p+(p-1)\sum_{k=1}^{q}n_k\right)
 =
 \sum_{k=1}^{q}\bar d_k+p\sum_{i=1}^{s}d_i-c.
 \tag{4.2}
\]

Indeed, a wild branch contributes
\[
 pm\left(1-\frac1{pt_k}+\frac{n_k(p-1)}p\right)
 =
 pm-\bar d_k+m n_k(p-1),
\]
whereas a tame branch contributes $p(m-d_i)$.  Since the total normalized
different is $3-c/(pm)$, summing these terms gives \textup{(4.2)}.

\begin{lemma}\label{lem:number-wild}
The number $q$ of wild branch points satisfies
\[
        1\leq q\leq2.
\]
\end{lemma}

\begin{proof}
Suppose first that $q=0$.  Let $B\leq G$ be the subgroup generated by
all inertia groups of the cover $X\to X/G$.  Since every inertia group
has order prime to $p$, and $G$ is cyclic, the order of $B$ is prime to
$p$.  Hence $B\neq G$.  By construction every inertia group is contained
in $B$, so the induced cover
\[
        X/B\longrightarrow X/G\cong\Pone
\]
is unramified and has degree $[G:B]>1$.  Riemann--Hurwitz would give
\[
        2g(X/B)-2=-2[G:B]<-2,
\]
which is impossible.  Thus $q\geq1$.

For a wild branch point, Lemma~\ref{lem:oddlocal} gives
\[
 \delta_{\rm wild}
 =
 1-\frac1{pt}+\frac{n(p-1)}p
 \geq
 2-\frac2p
 \geq\frac43.
\]
If $q\geq3$, then
\[
        \sum_Q\delta_Q\geq3\cdot\frac43=4,
\]
contrary to \textup{(2.2)}.  Therefore $q\leq2$.
\end{proof}

\subsection{Two wild branch points}

\begin{proposition}\label{prop:twowildodd}
If $q=2$, then necessarily
\[
        p=3,\qquad m=1,\qquad N=3,\qquad g=2,
\]
and $N=2g-1$.  In suitable coordinates
\[
        X:\quad y^3-y=ax+\frac{b}{x},
        \qquad ab\neq0.
        \tag{4.3}
\]
\end{proposition}

\begin{proof}
Since each $n_k\geq1$, \textup{(4.2)} and the inequalities
$\bar d_k\leq m$, $d_i\leq m/2$ give
\[
 (s-1)p+(p-1)(n_1+n_2)
 <
 2+\frac{ps}{2}.
\]
As $n_1+n_2\geq2$, this is impossible for $p\geq5$.  Hence $p=3$, and
then it forces
\[
        n_1=n_2=1.
\]

With $p=3$, equation \textup{(4.2)} becomes
\[
        \bar d_1+\bar d_2+3\sum_{i=1}^{s}d_i
        =
        m(3s+1)+c.
        \tag{4.4}
\]
If $s\geq1$, the left-hand side is at most
\[
        2m+\frac{3sm}{2},
\]
which is strictly smaller than $m(3s+1)+c$.  Hence $s=0$.

If $m>1$, the tame cover $X/S\to\Pone$ is ramified only over the two
wild base points.  Lemma~\ref{lem:lcm} gives $t_1=t_2=m$, hence
$\bar d_1=\bar d_2=1$.  Equation \textup{(4.4)} then gives
\[
        2=m+c,
\]
which is impossible.  Thus $m=1$, and again $\bar d_1=\bar d_2=1$.
Equation \textup{(4.4)} gives $2=1+c$, so $c=1$.  Hence
$N=3=2g-1$ and $g=2$.

The cover is an Artin--Schreier cover of $\Pone$ with exactly two simple
poles.  Moving them to $0$ and $\infty$ and reducing the equation gives
\textup{(4.3)}.
\end{proof}

\subsection{One wild branch point}

By Lemma~\ref{lem:number-wild}, after Proposition~\ref{prop:twowildodd}
the only remaining possibility is $q=1$.  Write
\[
        d_0=\frac mt,\qquad n=n_1.
\]
Equation \textup{(4.2)} becomes
\[
 m\bigl((s-2)p+n(p-1)\bigr)
 =
 d_0+p\sum_{i=1}^{s}d_i-c.
 \tag{4.5}
\]
Since $d_0\leq m$, $d_i\leq m/2$, and $n\geq1$, equation
\textup{(4.5)} gives
\[
 m\bigl((s-1)p-1\bigr)
 \leq
 m\bigl((s-2)p+n(p-1)\bigr)
 <
 m+\frac{psm}{2}.
\]
After division by $m$ this yields
\[
        s<2+\frac4p.
        \tag{4.6}
\]
Thus $s\leq2$ for $p\geq5$, while $s\leq3$ for $p=3$.

\begin{proposition}\label{prop:s0s1}
The case $s=0$ does not occur.  If $s=1$, the only possibility is
\[
        p=3,\qquad m=2,\qquad n=2,\qquad N=6,\qquad g=3,
\]
and it belongs to the boundary $N=2g$.
\end{proposition}

\begin{proof}
Suppose first that $s=0$.  The tame cover $X/S\to\Pone$ can ramify only
over the wild base point.  Hence it is trivial, so $m=1$ and $d_0=1$.
Equation \textup{(4.5)} becomes
\[
        n(p-1)-2p=1-c.
\]
For $c=2$ this would imply $p-1\mid1$, impossible.  For $c=1$ it implies
$p-1\mid2$, so $p=3$ and $n=3$, contradicting the fact that $n$ is a
reduced Artin--Schreier pole order and hence is prime to $p$.

Now let $s=1$.  If $t=1$, the tame cover $X/S\to\Pone$ would be ramified
at only one point.  Thus $t>1$, and Lemma~\ref{lem:lcm} applied to its
two branch points gives
\[
        t=e_1=m,\qquad d_0=d_1=1.
\]
Equation \textup{(4.5)} becomes
\[
        m\bigl(n(p-1)-p\bigr)=p+1-c.
        \tag{4.7}
\]
The coefficient on the left is negative for $n=1$.  If $n\geq3$, then
$m\geq2$ and
\[
 m\bigl(n(p-1)-p\bigr)
 \geq 2\bigl(3(p-1)-p\bigr)
 =2(2p-3)>p\geq p+1-c,
\]
so \textup{(4.7)} is impossible.  Thus $n=2$.

For $c=2$, equation \textup{(4.7)} is
\[
        m(p-2)=p-1,
\]
whose only solution with $m>1$ is $(p,m)=(3,2)$.  For $c=1$ one gets
$m(p-2)=p$; the only integral candidate is $(p,m)=(3,3)$, excluded by
$(p,m)=1$.
\end{proof}

The model in Proposition~\ref{prop:s0s1} follows from
Lemma~\ref{lem:fiberproduct}.  Put the two branch points of the
$C_2$-cover $X/S\to\Pone$ at $0,\infty$ and write $z=x^2$.  The
Artin--Schreier cover of the $z$-line has one reduced pole of order two
at infinity, so
\[
        y^3-y=a z^2+bz.
\]
Thus
\[
        X:\quad y^3-y=a x^4+b x^2,\qquad a\neq0.
        \tag{4.8}
\]

\begin{proposition}\label{prop:s2t1}
Assume $s=2$ and $t=1$.  Then the possibilities are:
\[
\begin{array}{c|c|c|c}
 & p&m&n\\
\hline
N=2g&3&4&1\\
N=2g-1&3&5&1\\
N=2g-1&5&3&1.
\end{array}
\tag{4.9}
\]
\end{proposition}

\begin{proof}
Since $t=1$, the tame cover $X/S\to\Pone$ is unramified at the wild
base point and ramifies only at the two tame branch points.  By
Lemma~\ref{lem:lcm},
\[
        e_1=e_2=m,
\]
so
\[
        d_0=m,\qquad d_1=d_2=1.
\]
Equation \textup{(4.5)} becomes
\[
        m\bigl(n(p-1)-1\bigr)=2p-c.
        \tag{4.10}
\]
If $n\geq2$, the left-hand side is at least
$2(2p-3)>2p-c$, so $n=1$.

For $c=2$,
\[
        m(p-2)=2p-2,
\]
hence
\[
        m=2+\frac{2}{p-2},
\]
and the only odd prime solution is $(p,m)=(3,4)$.

For $c=1$,
\[
        m(p-2)=2p-1,
\]
hence
\[
        m=2+\frac{3}{p-2}.
\]
Thus $p-2$ divides $3$, and the two solutions are
$(p,m)=(3,5)$ and $(5,3)$.
\end{proof}

These three cases have the models
\[
\begin{array}{ll}
p=3,\ m=4:&
y^3-y=\displaystyle\frac{\kappa}{x^4-1},\\[2mm]
p=3,\ m=5:&
y^3-y=\displaystyle\frac{\kappa}{x^5-1},\\[2mm]
p=5,\ m=3:&
y^5-y=\displaystyle\frac{\kappa}{x^3-1},
\end{array}
\qquad \kappa\neq0.
\tag{4.11}
\]
Indeed, take $z=x^m$ for the tame two-point cover and move the wild base
point to $z=1$.

\begin{proposition}\label{prop:s2tgt1}
Assume $s=2$ and $t>1$.  Then $n=1$.  For $N=2g$ there are precisely
three possibilities:
\[
\begin{array}{c|c|c|c|c}
p&g(X/S)&m&d_0&\{d_1,d_2\}\\
\hline
3&1&4&1&\{1,2\}\\
5&1&6&1&\{2,3\}\\
3&2&10&1&\{2,5\}.
\end{array}
\tag{4.12}
\]
There is no possibility for $N=2g-1$.
\end{proposition}

\begin{proof}
Let
\[
        h=g(X/S).
\]
The tame cover $X/S\to\Pone$ has three branch points of indices
$t,e_1,e_2$.  In particular $d_0,d_1,d_2$ are pairwise coprime by
Lemma~\ref{lem:lcm}.

Since $t,e_1,e_2\geq2$, one has
$d_0,d_1,d_2\leq m/2$.  Equation \textup{(4.5)} therefore gives
\[
        n(p-1)<p+\frac12.
\]
For $p\geq3$ this forces $n=1$.

Riemann--Hurwitz for the tame $C_m$-cover gives
\[
        2h-2=m-d_0-d_1-d_2.
        \tag{4.13}
\]
Using \textup{(4.5)} with $n=1$ to eliminate $d_1+d_2$ yields
\[
        m=p(2h-2)+(p-1)d_0+c.
        \tag{4.14}
\]

If $h=0$, a nontrivial tame cyclic group acting on a rational curve has
exactly two branch points, whereas the present cover has three.  Hence
$h\neq0$.

Let $h=1$, and let $\alpha$ generate the $C_m$-action on the elliptic
curve $X/S$.  Write
\[
        \alpha=\tau_Q\circ\beta,
\]
where $\tau_Q$ is translation by $Q$ and $\beta$ fixes the origin.  Since
the action is not semiregular, $\beta\neq1$: otherwise $\alpha$ would be
a nontrivial translation and the cyclic action would be fixed-point free.
The endomorphism $1-\beta$ is then a nonzero isogeny and hence is
surjective.  Thus there exists $P$ with $(1-\beta)(P)=Q$, equivalently
$\alpha(P)=P$.  Taking $P$ as the origin, $\alpha$ becomes an automorphism
of the elliptic curve fixing the origin.  Since $(m,p)=1$, the standard
classification gives
\[
        m\in\{2,3,4,6\};
\]
see \cite[Theorem~III.10.1]{Silverman2009}.
Equation \textup{(4.14)} becomes
\[
        m=(p-1)d_0+c.
\]
For $c=2$, the only solutions with $(p,m)=1$ and $d_0\mid m$ are
\[
        (p,m,d_0)=(3,4,1),\qquad(5,6,1).
\]
Equation \textup{(4.13)} then gives
\[
        d_1+d_2=3,\qquad d_1+d_2=5,
\]
respectively, and the divisor conditions give
\[
        \{d_1,d_2\}=\{1,2\},\qquad\{2,3\}.
\]
For $c=1$ there is no solution.

Assume finally $h\geq2$.  Since $C_m$ is abelian,
\[
        m\leq4h+4
\]
by \cite[Theorem~11.79]{HKT2008}.
From \textup{(4.14)} and $d_0\geq1$, this is impossible for $p\geq5$.
Hence $p=3$, and
\[
        m=6h-6+2d_0+c.
        \tag{4.15}
\]
If $c=2$, the abelian bound gives
\[
        h+d_0\leq4.
\]
The possibilities with $h\geq2$ are
\[
 (h,d_0)=(2,1),(2,2),(3,1).
\]
Equation \textup{(4.15)} gives $m=10,12,16$, respectively.  The middle
case is excluded by $(m,3)=1$.  In the last case $d_0=1$, so $t=m$ and
$C_m$ fixes a point of the genus-three curve $X/S$.  Since
$p\nmid m$, the point-fixing bound
\[
        m\leq4h+2
\]
from \cite[Theorem~11.60]{HKT2008} gives $16\leq14$, a contradiction.  Thus
$(h,m,d_0)=(2,10,1)$.  Equation \textup{(4.13)} gives
$d_1+d_2=7$, and the proper divisors of $10$ force
$\{d_1,d_2\}=\{2,5\}$.

If $c=1$, equation \textup{(4.15)} becomes
\[
        m=6h-5+2d_0.
\]
Together with $m\leq4h+4$ this gives
\[
        h+d_0\leq4.
\]
Since $h\geq2$ and $d_0\geq1$, the only possibilities are
\[
        (h,d_0)=(2,1),(2,2),(3,1),
\]
and they give respectively
\[
        m=9,11,15.
\]
The first and third are divisible by $3$, contrary to $(m,3)=1$, while
for $m=11$ one would have $d_0=2$, impossible because $d_0=m/t$ is a
divisor of $m$.  Hence no case remains.
\end{proof}

The three tame covers in \textup{(4.12)} may be written as
\[
\begin{array}{ll}
m=4:& w^4=z(z-1)^2,\\
m=6:& w^6=z^2(z-1)^3,\\
m=10:& w^{10}=z^2(z-1)^5.
\end{array}
\tag{4.16}
\]
We verify the Kummer data explicitly.  For $m=4$ the inertia indices are
$4,2,4$.  The inertia-two residue is $2$, while the two full-inertia
residues lie in $\{1,3\}$.  The zero-sum condition forces them to be
equal, and multiplication by $3$ identifies the two possibilities; thus
we may take $(1,2,1)$.  For $m=6$ the indices are $3,2,6$.
Normalize the full-inertia residue to $1$.  The inertia-two residue is
$3$, and the zero-sum condition then forces the inertia-three residue to
be $2$, giving $(2,3,1)$.  For $m=10$ the indices are $5,2,10$.
After normalizing the full-inertia residue to $1$, the inertia-two
residue is $5$ and the zero-sum condition forces the inertia-five
residue to be $4$.  Multiplying $(4,5,1)$ by $3$ gives the equivalent
vector $(2,5,3)$ appearing in \textup{(4.16)}.  Hence the displayed
three equations are, up to units and permutations of branch points, the
unique Kummer normal forms with the required inertia data.  In each
equation the point at infinity has full inertia $m$.

By Lemma~\ref{lem:fiberproduct}, the Artin--Schreier equation is
\[
        y^p-y=a z,\qquad a\neq0.
        \tag{4.17}
\]

\begin{proposition}\label{prop:s3}
The case $s=3$ does not occur.
\end{proposition}

\begin{proof}
By \textup{(4.6)}, this can occur only for $p=3$.  Equation
\textup{(4.5)} reads
\[
        m(2n+3)=d_0+3(d_1+d_2+d_3)-c.
        \tag{4.18}
\]
Since $d_0\leq m$ and $d_i\leq m/2$, the right-hand side is smaller than
$11m/2$.  Hence $2n+3<11/2$, so $n=1$.

Put $D=d_1+d_2+d_3$ and $h=g(X/S)$.  Riemann--Hurwitz gives
\[
        2h-2=2m-d_0-D.
\]
Using \textup{(4.18)} with $n=1$ gives again
\[
        m=6h-6+2d_0+c.
        \tag{4.19}
\]

If $h=0$, a tame cyclic group on a rational curve cannot have the three
tame branch points already present here.  If $h=1$, apply the preceding
elliptic-curve argument to a generator of $C_m$: writing it as a
translation followed by an automorphism fixing the origin shows that the
generator itself has a fixed point.  Hence
$m\in\{2,3,4,6\}$ by
\cite[Theorem~III.10.1]{Silverman2009}.  Since $(m,3)=1$, only $m=2,4$ remain.  For $c=2$,
\textup{(4.19)} leaves $m=4,d_0=1$, but then
\[
        D=2m-d_0=7,
\]
whereas three proper divisors of $4$ have sum at most $6$.  For $c=1$
there is no integral $d_0$.

Let $h\geq2$.  This argument also includes the subcase $t=1$, for
which $d_0=m$.  Since $C_m$ is abelian, the bound $m\leq4h+4$ together
with \textup{(4.19)} gives a short finite list.

If $c=1$, then
\[
        m=6h-5+2d_0,
\]
so $h+d_0\leq4$.  Thus
\[
        (h,d_0)=(2,1),(2,2),(3,1),
\]
which give $m=9,11,15$, respectively.  The values $9$ and $15$ are
excluded by $(m,3)=1$, while $m=11$ would require the divisor
$d_0=2$ of $11$, impossible.  Hence $c=1$ gives no solution.

If $c=2$, then
\[
        m=6h-4+2d_0,
\]
and again $h+d_0\leq4$.  The same three pairs give
$m=10,12,16$.  The value $12$ is excluded by $(m,3)=1$.  For
$m=16$ one has $(h,d_0)=(3,1)$, so $t=m$ and $C_m$ fixes a point of the
genus-three curve $X/S$; the point-fixing bound
$m\leq4h+2$ gives $16\leq14$, a contradiction.  Hence only
$(h,m,d_0)=(2,10,1)$ remains.  But then
\[
        D=2m-d_0-(2h-2)=17,
\]
whereas each $d_i$ is a proper divisor of $10$ and hence at most $5$.
Thus $D\leq15$, a contradiction.
\end{proof}

Combining the preceding propositions gives the odd-characteristic
classification.

\begin{theorem}[Odd-characteristic wild boundary]\label{thm:oddwild}
Assume $p\geq3$ and $p\mid N$, with
$N\in\{2g,2g-1\}$.  Then, up to coordinate changes and
Artin--Schreier transformations, the possibilities are exactly:
\[
\begin{array}{c c c l}
\toprule
p & N & g & \text{model}\\
\midrule
3 & 3 & 2 &
 y^3-y=ax+b/x,\quad ab\neq0,\\[1mm]
3 & 6 & 3 &
 y^3-y=ax^4+bx^2,\quad a\neq0,\\[1mm]
3 & 12 & 6 &
 y^3-y=\kappa/(x^4-1),\quad \kappa\neq0,\\[1mm]
3 & 12 & 6 &
 w^4=z(z-1)^2,\quad y^3-y=a z,\quad a\neq0,\\[1mm]
5 & 30 & 15 &
 w^6=z^2(z-1)^3,\quad y^5-y=a z,\quad a\neq0,\\[1mm]
3 & 30 & 15 &
 w^{10}=z^2(z-1)^5,\quad y^3-y=a z,\quad a\neq0,\\[1mm]
3 & 15 & 8 &
 y^3-y=\kappa/(x^5-1),\quad \kappa\neq0,\\[1mm]
5 & 15 & 8 &
 y^5-y=\kappa/(x^3-1),\quad \kappa\neq0.
\\
\bottomrule
\end{array}
\tag{4.20}
\]
The first and last two rows belong to $N=2g-1$; the remaining rows belong
to $N=2g$.
\end{theorem}

\begin{proof}
Completeness is the content of Propositions
\ref{prop:twowildodd}--\ref{prop:s3}.  It remains only to verify the
converse.

All Kummer degrees in \textup{(4.20)} are prime to $p$, and all
Artin--Schreier equations have reduced pole orders.  The two extensions
over the common rational base have coprime degrees, so their compositum
has Galois group
\[
        C_m\times C_p\cong C_{pm}.
\]
The genus follows directly from Riemann--Hurwitz.

For example, in the three fiber-product cases the tame covering curves
have genera $1,1,2$, respectively.  The Artin--Schreier function has one
pole on the tame covering curve, of orders $4,6,10$, respectively.
Thus
\[
\begin{array}{c|c}
(p,m)&2g-2\\
\hline
(3,4)&3(0)+2(4+1)=10,\\
(5,6)&5(0)+4(6+1)=28,\\
(3,10)&3(2)+2(10+1)=28.
\end{array}
\]
Hence the genera are $6,15,15$.  The remaining rows are checked in the
same way:
two simple poles for $p=3$ give genus $2$; one pole of order four gives
genus $3$; four simple poles in characteristic $3$ give genus $6$; five
simple poles in characteristic $3$ and three simple poles in
characteristic $5$ both give genus $8$.
\end{proof}

\begin{remark}
For every pair $(X,G)$ occurring in Theorem~\ref{thm:oddwild}, the
preceding reductions give a $G$-equivariant model belonging to one of the
displayed families, after the stated coordinate changes on the rational
base and changes of Artin--Schreier or Kummer generator.  This is the
sense in which the list is complete.  No claim is made that distinct
values of the displayed parameters define nonisomorphic pairs.
\end{remark}

\section{Characteristic two}

Assume throughout this section that $\Char(K)=2$.
The ramification estimates and the numerical analysis in this section are
closely related to those used by Timpanella in \cite[Section~3]{Timpanella2026}
for the range $N\geq 2g+1$.  At the boundary value $N=2g$, several of the
inequalities occurring there become equalities and lead to the cases below.
For completeness, we keep the short arguments needed here and indicate the
corresponding results in \cite{Timpanella2026}.  The Artin--Schreier--Witt
normal forms in Theorems~\ref{thm:C4normal} and \ref{thm:C12normal} are not
part of that classification, since \cite{Timpanella2026} proves that
$4\nmid N$ in the range $N\geq2g+1$.

If $N=2g-1$, then $N$ is odd and the action is tame.  Thus
Theorem~\ref{thm:tame} gives exactly
\[
        (N,g)=(3,2),\qquad(3,3,3,3),
\]
and
\[
        (N,g)=(15,8),\qquad(3,3,5,5).
\]
We therefore assume $N=2g$ and write
\[
        N=2^a m,\qquad m\ \text{odd}.
\]

Lemma~\ref{lem:oddlocal}, applied with $p=2$, shows that if a cyclic
stabilizer has order $2t$, with $t$ odd, then
\[
        \delta\geq\frac32-\frac1{2t}.
        \tag{5.1}
\]
We shall also need the following two elementary consequences of
Artin--Schreier--Witt ramification theory; compare
\cite[Lemmas~2.9 and~2.10]{Timpanella2026}.

\begin{lemma}\label{lem:char2local}
Let $I$ be a cyclic inertia group in characteristic two.
\begin{enumerate}[label=\textup{(\roman*)}]
\item If $|I|=4t$, with $t$ odd, then
\[
        \delta\geq\frac94-\frac1{4t}.
        \tag{5.2}
\]
More precisely, if $j_1<j_2$ are the positive lower breaks, then
$j_1\geq t$ and $j_2\geq3j_1$.
\item If $8\mid |I|$, then $\delta\geq7/2$.
\end{enumerate}
\end{lemma}

\begin{proof}
Assume first that $|I|=4t$.  For positive lower indices, the ramification
groups have orders four and then two.  The same local-parameter
commutation argument as in Lemma~\ref{lem:oddlocal} gives $t\mid j_1$,
hence $j_1\geq t$.  After quotienting by the tame subgroup of order
$t$, Herbrand's function gives the upper breaks
\[
        u_1=\frac{j_1}{t},\qquad
        u_2=\frac{j_1+j_2}{2t}.
\]
For a cyclic $C_4$-extension one has $u_2\geq2u_1$; see, for example,
\cite{Schmid1937,ElderKeating2026}.  Hence $j_2\geq3j_1$.  Hilbert's different
formula now gives
\[
 d=(4t-1)+3j_1+(j_2-j_1)
   =4t-1+2j_1+j_2
   \geq9t-1,
\]
and division by $4t$ proves \textup{(5.2)}.

Suppose next that $8\mid |I|$.  Let $H\leq I$ be the subgroup of
index eight and let $E=L^H$.  Transitivity of the different gives
\[
 \frac{d(L/K)}{|I|}
 \geq \frac{d(E/K)}8.
\]
Thus it is enough to treat the cyclic degree-eight quotient.  For a
cyclic $C_8$-extension the three upper
breaks satisfy
\[
        u_1\geq1,\qquad u_2\geq2u_1,\qquad u_3\geq2u_2;
\]
see \cite{Schmid1937,ElderKeating2026}.
Thus the smallest possible upper breaks are $(1,2,4)$, corresponding to
lower breaks $(1,3,11)$.  The different exponent is therefore at least
\[
        2(8-1)+2(4-1)+8(2-1)=28.
\]
Consequently the normalized different of the degree-eight quotient is at
least $28/8=7/2$, and the same lower bound holds for $I$.
\end{proof}

\begin{proposition}\label{prop:2part}
One has $a\leq2$.
\end{proposition}

\begin{proof}
This is the boundary analogue of the first reduction in
\cite[Proposition~3.1]{Timpanella2026}.  By
Proposition~\ref{prop:rationalquotient} and
Lemma~\ref{lem:generate-stab}, some stabilizer contains the full Sylow
$2$-subgroup of $G$.  If $a\geq3$, Lemma~\ref{lem:char2local}(ii) shows
that the corresponding branch point contributes at least $7/2$, whereas
the total contribution is
\[
        3-\frac1{2^{a-1}m}<3.
\]
This is impossible.
\end{proof}

\subsection{The case $a=1$}

Now
\[
        N=2m=2g,\qquad g=m,
\]
so $m\geq3$ is odd.  Let
\[
        P=C_2,\qquad B=C_m,\qquad
        Y=X/P,\qquad Z=X/G\cong\Pone.
\]

\begin{proposition}\label{prop:two-wild-char2}
If $X\to Z$ has two wild branch points, then it has no other branch
point, both wild stabilizers have order $2m$, and
\[
        X:\quad y^2+y=c(x^m+x^{-m}),
        \qquad c\in K^\times.
        \tag{5.3}
\]
\end{proposition}

\begin{proof}
This argument is the boundary equality case of the analysis in
\cite[Lemma~3.2]{Timpanella2026}.

There are at most two wild branch points because each contributes at
least one and the total is smaller than three.

Suppose that there are two.  If there were a tame branch point and one
of the wild inertia groups had odd part $t>1$, then by \textup{(5.1)}
the three contributions would be at least
\[
        \frac43+1+\frac23=3,
\]
a contradiction.  If both odd parts were trivial, then the Kummer cover
$Y\to Z$ would be unramified at the two wild points and, since $m>1$,
would have at least two tame branch points by Lemma~\ref{lem:lcm}.
The total contribution would then exceed three.  Hence there are no tame
branch points.

The Kummer cover $Y\to Z$ can ramify only over the two wild base points.
It is nontrivial, so Lemma~\ref{lem:lcm} gives full ramification at both:
their odd inertia parts are $m$.  Write the lower breaks as $mn_1,mn_2$,
where $n_1,n_2$ are positive odd integers.  Riemann--Hurwitz gives
\[
 2m-2
 =
 -4m+\bigl(2m-1+mn_1\bigr)
       +\bigl(2m-1+mn_2\bigr),
\]
hence
\[
        n_1+n_2=2.
\]
Thus $n_1=n_2=1$.

The two-point Kummer cover is rational; choose a coordinate $x$ on $Y$
so that the quotient coordinate is $z=x^m$.  The quadratic subcover of
the $z$-line has two simple poles, which may be placed at $0,\infty$:
\[
        y^2+y=a z+\frac{b}{z},
        \qquad ab\neq0.
\]
Choose $\mu\in K^\times$ with
\[
        \mu^{2m}=\frac{b}{a}.
\]
After the change $x\mapsto\mu x$, the two coefficients are equal.  This
gives \textup{(5.3)}.

Conversely, for any odd $m\geq3$ and $c\in K^\times$, the equation
\[
        y^2+y=c(x^m+x^{-m})
\]
defines a quadratic Artin--Schreier cover of the $x$-line with exactly
two poles, at $0$ and $\infty$, both of reduced order $m$.  Hence
\[
        2g(X)-2=-4+2(m+1)=2m-2,
\]
so $g(X)=m$.  The transformations
\[
        x\longmapsto\zeta x\quad(\zeta^m=1),
        \qquad y\longmapsto y+1
\]
commute and generate $C_m\times C_2\cong C_{2m}$, since $m$ is odd.
Thus every member of \textup{(5.3)} has the required cyclic group of
order $2m=2g$.
\end{proof}

Assume now that there is exactly one wild branch point $Q_0$.

\begin{lemma}\label{lem:eq-char2}
The following is the equality case of
\cite[Lemma~2.11]{Timpanella2026}, included here for convenience.
Let $d_0,d_1,d_2$ be pairwise coprime divisors of an odd integer $m$,
and assume $m/d_i\geq3$ for all $i$.  If
\[
        m+2=d_0+2d_1+2d_2,
        \tag{5.4}
\]
then, up to exchanging $d_1,d_2$,
\[
        (m;d_0,d_1,d_2)
        =
        (3;1,1,1)
\]
or
\[
        (15;1,3,5).
\]
\end{lemma}

\begin{proof}
Since $d_0,d_1,d_2$ are pairwise coprime divisors of $m$, their product
divides $m$.

Assume first that $d_0,d_1,d_2>1$.  Since all three integers are odd,
$d_i\geq3$.  Then
\[
\begin{split}
 d_0d_1d_2-d_0-2d_1-2d_2+2
 &=
 d_0(d_1d_2-1)-2(d_1+d_2-1)\\
 &\geq
 3(d_1d_2-1)-2(d_1+d_2-1)>0.
\end{split}
\]
As $d_0d_1d_2\mid m$, this contradicts \textup{(5.4)}.

Suppose next that exactly one of $d_1,d_2$ is equal to one; by symmetry
take $d_1=1<d_2$.  Equation \textup{(5.4)} becomes
\[
        m=d_0+2d_2.
\]
If $d_0>1$, then $d_0d_2\mid m$, so
\[
        d_0d_2\mid d_0+2d_2.
\]
Reducing modulo $d_0$ and using $\gcd(d_0,d_2)=1$ gives $d_0\mid2$,
impossible because $d_0$ is odd and greater than one.  Thus $d_0=1$.
Now $m=2d_2+1$, and since $d_2\mid m$ we obtain $d_2\mid1$, again a
contradiction.  Hence this case does not occur.

If $d_1=d_2=1$, then \textup{(5.4)} gives
\[
        m=d_0+2.
\]
Since $d_0\mid m$, we have $d_0\mid2$.  Thus $d_0=1$ and $m=3$.

It remains to assume $d_0=1$ and $d_1,d_2>1$.  Equation \textup{(5.4)}
gives
\[
        m=2d_1+2d_2-1.
\]
Since $d_1d_2\mid m$,
\[
        d_1d_2\leq 2d_1+2d_2-1,
\]
or equivalently
\[
        (d_1-2)(d_2-2)\leq3.
\]
The integers $d_1,d_2$ are odd and at least $3$.  Hence, up to order,
the only possibilities are
\[
        (d_1,d_2)=(3,3)\quad\text{or}\quad(3,5).
\]
The first is excluded by $\gcd(d_1,d_2)=1$.  The second gives
$m=15$.  This proves the assertion.
\end{proof}

\begin{proposition}\label{prop:onewild-char2}
Assume that $X\to Z$ has exactly one wild branch point.  Then there are
exactly two tame branch points, and, up to exchanging them, one of the
following occurs:
\[
        (m;d_0,d_1,d_2)=(3;1,1,1)
\]
or
\[
        (m;d_0,d_1,d_2)=(15;1,3,5).
        \tag{5.5}
\]
The corresponding curves may be written as
\[
\begin{cases}
 w^3=z(z-1),\\
 y^2+y=a z,
\end{cases}
\qquad (N,g)=(6,3),
\tag{5.6}
\]
or
\[
\begin{cases}
 w^{15}=z^3(z-1)^5,\\
 y^2+y=a z,
\end{cases}
\qquad (N,g)=(30,15),
\tag{5.7}
\]
with $a\neq0$.
\end{proposition}

\begin{proof}
A tame branch point contributes at least $2/3$, while the wild point
contributes at least one, so there are at most two tame branch points.

There cannot be no tame branch point.  Indeed, if the odd part $t$ of
the wild inertia is one, then the degree-$m$ Kummer cover $Y\to Z$ is
unramified, which is impossible for a nontrivial cover of $\Pone$.  If
$t>1$, then $Y\to Z$ is ramified over exactly one point, again impossible
by Lemma~\ref{lem:lcm}.  Hence at least one tame branch point occurs.

There cannot be exactly one tame branch point.  If $t=1$, the Kummer
cover $Y\to Z$ would ramify at only one point.  Thus $t>1$, and
Lemma~\ref{lem:lcm} forces full ramification $m$ at the wild and tame
base points.  Let $j=mn$, with $n$ odd.  The wild point has inertia
order $2m$ and contributes different exponent
\[
        (2m-1)+j,
\]
while the tame branch point contributes $2(m-1)$.  Hence
Riemann--Hurwitz gives
\[
\begin{split}
        2g-2
        &=-4m+(2m-1+j)+2(m-1)\\
        &=j-3.
\end{split}
\]
This is the boundary version of equation (3.1) in the proof of
\cite[Theorem~1.1]{Timpanella2026}.
Since $g=m$, this becomes $j=2m+1$, impossible because $m\mid j$ and
$m\geq3$.

Hence there are two tame branch points.  If $t=1$, the Kummer cover is
unramified at the wild point and ramified only at the two tame points,
both with index $m$.  If $j=n$ is the wild break, then the wild point
has inertia order $2$ and contributes $m(n+1)$ to the global different,
whereas each tame branch point contributes $2(m-1)$.  Thus
\[
\begin{split}
        2g-2
        &=-4m+m(n+1)+4(m-1)\\
        &=2m-4+m(n-1).
\end{split}
\]
Compare equation (3.2) in the proof of
\cite[Theorem~1.1]{Timpanella2026}.
Since $g=m$, one gets $m(n-1)=2$, impossible for odd $m\geq3$.

Thus $t>1$.  Let
\[
        d_0=\frac mt,\qquad d_i=\frac m{e_i}\quad(i=1,2).
\]
By Lemma~\ref{lem:lcm}, the three $d_i$ are pairwise coprime.
Write the wild break as $j=tn$, with $n$ positive and odd.
At the wild branch point the inertia group has order $2t$; its local
different exponent is $2t-1+j$, and there are $d_0=m/t$ points above
the corresponding point of $Z$.  The two tame branch points contribute
$2(m-d_1)$ and $2(m-d_2)$, respectively.  Therefore
Riemann--Hurwitz gives
\[
\begin{split}
 2m-2
 &= -4m+d_0(2t-1+j)
       +2(m-d_1)+2(m-d_2)\\
 &= 2m-2d_1-2d_2+d_0(j-1).
\end{split}
\]
Since $d_0t=m$ and $j=tn$, this is equivalent to
\[
        mn+2=d_0+2d_1+2d_2.
        \tag{5.8}
\]
Since $t,e_1,e_2$ are odd and greater than one, they are at least
three.  Hence
\[
        d_0,d_1,d_2\leq\frac m3,
        \qquad
        d_0+2d_1+2d_2\leq\frac{5m}{3}.
\]
The integer $n$ is positive and odd.  If $n\geq3$, then
\[
        mn+2\geq3m+2>\frac{5m}{3},
\]
contradicting \textup{(5.8)}.  Thus $n=1$, so \textup{(5.8)} becomes
\[
        m+2=d_0+2d_1+2d_2.
\]
Lemma~\ref{lem:eq-char2} now gives exactly the two cases in
\textup{(5.5)}.

For the two numerical possibilities we now verify the three-point
Kummer exponent data.  When $m=3$, all three inertia indices are $3$,
so every exponent residue is $1$ or $2$.  If $k$ of the three residues
are $1$, their sum is congruent to $-k$ modulo $3$; the zero-sum
condition therefore gives $k=0$ or $3$.  Multiplication by the unit $2$
identifies these two cases, so the vector may be taken to be
$(1,1,1)$.

When $m=15$, the inertia indices are $15,5,3$.  Normalize the
full-inertia residue to $1$.  The inertia-five residue is one of
$3,6,9,12$, while the inertia-three residue is $5$ or $10$.  The
zero-sum condition has the unique solution
\[
        (1,9,5)
\]
with this normalization.  Multiplication by $7$ gives the equivalent
vector $(7,3,5)$, i.e. $(3,5,7)$ after ordering the finite branch points
first and the point at infinity last.  This proves the claimed
uniqueness in both cases.

Thus for $m=3$ the tame cover has three totally ramified points and is
represented by $w^3=z(z-1)$.  Put the wild base point at infinity.
The quadratic subcover has a simple pole there, hence $y^2+y=az$.
The Kummer curve has genus one, and $z$ has a pole of order three at the
unique point above infinity.  Thus the Artin--Schreier cover has one pole
of order three, and Riemann--Hurwitz gives $g(X)=3$.

For $m=15$, the three tame ramification indices are $15,5,3$.
The Kummer equation
\[
        w^{15}=z^3(z-1)^5
\]
has precisely these indices, with full ramification at infinity.
Riemann--Hurwitz for this degree-$15$ Kummer cover gives
\[
        2h-2=-30+(15-3)+(15-5)+(15-1)=6,
\]
so its genus is $h=4$.  Moreover, $z$ has a pole of order fifteen at the
unique point above infinity.  Again the quadratic subcover is $y^2+y=az$, and
Riemann--Hurwitz gives
\[
        2g(X)-2=2(2\cdot4-2)+(15+1)=28,
\]
so $g(X)=15$.  In both cases the two extensions over $K(z)$ have
coprime degrees, so their compositum has cyclic Galois group of order
$2m$.
\end{proof}

\subsection{The case $a=2$}

We now assume
\[
        N=4m=2g,
        \qquad g=2m,
\]
with $m$ odd.

\begin{proposition}\label{prop:a2}
One has
\[
        m=1\quad\text{or}\quad m=3.
\]
Thus
\[
        (N,g)=(4,2)\quad\text{or}\quad(12,6).
\]
\end{proposition}

\begin{proof}
The same numerical estimate appears in
\cite[Proposition~3.1]{Timpanella2026}; here the boundary equality
$N=2g$ leaves the exceptional case $m=3$.

Let $Q_0$ be a branch point whose inertia contains the full $C_4$.
By \textup{(5.2)},
\[
        \delta_{Q_0}\geq\frac94-\frac1{4t}.
\]
There is no second wild branch point, since it would contribute at least
one and the total in \textup{(2.2)} is smaller than three.

Assume $m>1$.  Consider the tame quotient by the $C_4$-subgroup.  It is
a nontrivial cyclic cover of $\Pone$ of degree $m$.  Since the wild
branch contributes at least $2$ by \textup{(5.2)}, two tame branch points
would contribute at least another $2\cdot(2/3)$, and
\[
        2+2\cdot\frac23>3.
\]
Thus there is room for at most one tame branch point of $X\to X/G$.
If the odd part of the inertia at $Q_0$ were trivial, then this
degree-$m$ tame cover would have at most one branch point, impossible by
Lemma~\ref{lem:lcm}.  Hence it ramifies at $Q_0$, and it must ramify at
one further point.  By Lemma~\ref{lem:lcm}, a cyclic tame cover of
$\Pone$ with exactly two branch points is totally ramified at both.
Thus the stabilizers for $X\to X/G$ have orders $4m$ and $m$.

Using \textup{(5.2)} and the tame contribution $1-1/m$, we obtain
\[
\begin{split}
 2g-2
 &\geq
 -8m+
 4m\left(\frac94-\frac1{4m}+1-\frac1m\right)\\
 &=5m-5.
\end{split}
\]
Since $g=2m$, the left-hand side is $4m-2$.  Hence $m\leq3$.  As $m$
is odd and greater than one, $m=3$.  Together with the case $m=1$ this
proves the result.
\end{proof}

\subsection{The $C_4$ family in genus two}

We use the Artin--Schreier--Witt map
\[
        \wpw=F-\operatorname{id}:W_2(K(x))\longrightarrow W_2(K(x)).
\]
Every class has a reduced representative, and reduced representatives
modulo the natural $(\mathbb Z/4\mathbb Z)^\times$-action classify
$C_4$-covers of the fixed projective line; see
\cite[Proposition~2.4]{DangHippold2024}.  We shall also use explicitly
the length-two Witt addition law in characteristic two:
\[
 (u_0,u_1)+(v_0,v_1)
 =
 (u_0+v_0,\;u_1+v_1+u_0v_0).
 \tag{5.9}
\]
In particular,
\[
        \wpw(0,h)=(0,h^2+h).
        \tag{5.10}
\]
Since $K$ is algebraically closed, every constant Witt vector belongs to
$\wpw(W_2(K))$.  Indeed, for a prescribed constant vector $(A,B)$ one
first chooses $r\in K$ with $r^2+r=A$ and then chooses the second
coordinate so that $\wpw(r,s)=(A,B)$.

\begin{theorem}\label{thm:C4normal}
Let $\Char(K)=2$.  A pair $(X,G)$ with
\[
        g(X)=2,\qquad G\cong C_4
\]
is, after a coordinate change on $X/G$, represented by
\[
        \wpw(Y_0,Y_1)=(x,\lambda x^3),
        \qquad \lambda\in K^\times.
        \tag{5.11}
\]
No uniqueness of the parameter $\lambda$ is asserted here.
\end{theorem}

\begin{proof}
There is exactly one branch point.  Indeed, the branch point whose
stabilizer contains $C_4$ contributes at least two by \textup{(5.2)},
whereas the total contribution is
\[
        3-\frac12=\frac52;
\]
any second wild branch point would contribute at least one.

Let $j_1<j_2$ be the lower jumps.  By
Lemma~\ref{lem:char2local}(i),
\[
        j_1\geq1,\qquad j_2\geq3j_1.
\]
The different exponent is
\[
        d=3+2j_1+j_2.
\]
Since Riemann--Hurwitz gives $d=10$, we obtain
\[
        (j_1,j_2)=(1,5).
\]
The corresponding upper jumps are
\[
        (u_1,u_2)=(1,3).
\]

Move the unique branch point to infinity.  Here a length-two Witt
vector is called \emph{reduced} if, at every pole, neither component has
a polar term whose exponent is divisible by two.  On $\Pone$ the usual
successive Witt-coboundary reductions can be performed at all poles
simultaneously.  For a reduced length-two Witt vector, the upper-break
formula of \cite[Theorem~2.3]{ElderKeating2026} gives
\[
        u_1=\deg(f_0),\qquad
        u_2=\max\{2u_1,\deg(f_1)\},
\]
where only odd exponents occur in the reduced polynomial parts.  Hence
a reduced representative has the form
\[
        (a x,\;b x^3+c x),
        \qquad a,b\in K^\times,\quad c\in K,
        \tag{5.12}
\]
after removing constant terms.

We now record the effect of an affine change of coordinate
\[
        x\longmapsto\alpha x+\beta,\qquad \alpha\neq0.
\]
Substitution in \textup{(5.12)} gives
\[
\begin{split}
 F_0&=a\alpha x+A,\qquad A=a\beta,\\
 F_1&=b\alpha^3x^3+b\alpha^2\beta x^2
      +\alpha(b\beta^2+c)x+C
\end{split}
\]
for a constant $C\in K$.  Because every constant Witt vector is a
coboundary, we may add a constant vector $(A,\gamma)\in
\wpw(W_2(K))$ and choose $\gamma$ so that the new constant term in the
second coordinate vanishes.  By the addition law \textup{(5.9)}, the
result is
\[
 \left(
 a\alpha x,\;
 b\alpha^3x^3+b\alpha^2\beta x^2
 +\alpha(c+b\beta^2+a^2\beta)x
 \right).
 \tag{5.13}
\]
Next add $\wpw(0,\eta x)$, with
\[
        \eta^2=b\alpha^2\beta.
\]
Using \textup{(5.10)}, this removes the $x^2$-term and yields
\[
\begin{split}
 a'&=a\alpha,\\
 b'&=b\alpha^3,\\
 c'&=\alpha(c+b\beta^2+a^2\beta)+\eta.
\end{split}
\tag{5.14}
\]

Choose $\alpha=a^{-1}$ and write
\[
        \lambda=\frac{b}{a^3},\qquad
        C_0=\frac ca,\qquad
        t=a\beta.
\]
Then
\[
        b'=\lambda,\qquad
        c'=C_0+\lambda t^2+t+\sqrt{\lambda t}.
        \tag{5.15}
\]
Let $s=\sqrt{\lambda t}$.  Since $t=s^2/\lambda$, the condition
$c'=0$ is equivalent to
\[
        s^4+s^2+\lambda s+\lambda C_0=0.
        \tag{5.16}
\]
The field $K$ is algebraically closed, so \textup{(5.16)} has a
solution $s\in K$.  Choosing the corresponding $t$ and hence $\beta$
gives $c'=0$.  Thus every cover is represented by
\[
        (x,\lambda x^3).
\]

Conversely, let $\lambda\neq0$.  The first component $x$ has a simple
pole at infinity and is not an Artin--Schreier coboundary.  Hence the
Witt class $(x,\lambda x^3)$ has exact order four and defines a connected
$C_4$-cover of $\Pone$.  It is reduced and has upper jumps
\[
        u_1=1,\qquad u_2=\max\{2,3\}=3.
\]
Thus the lower jumps are
\[
        j_1=1,\qquad j_2=1+2(3-1)=5,
\]
and the different exponent at the unique branch point is
\[
        d=(1+1)(4-1)+(5-1)(2-1)=10.
\]
Riemann--Hurwitz gives
\[
        2g(X)-2=-8+10=2,
\]
so $g(X)=2$.  Therefore every $\lambda\in K^\times$ in
\textup{(5.11)} indeed produces a pair of the stated type.
\end{proof}

\subsection{The $C_{12}$ family in genus six}

\begin{theorem}\label{thm:C12normal}
Let $\Char(K)=2$.  A pair $(X,G)$ with
\[
        g(X)=6,\qquad G\cong C_{12}
\]
is the compositum, over a rational function field $K(z)$, of
\[
        x^3=z
        \tag{5.17}
\]
and
\[
        \wpw(Y_0,Y_1)
        =
        \left(\frac1z,\frac{\lambda}{z}\right),
        \qquad \lambda\in K.
        \tag{5.18}
\]
No uniqueness of the parameter $\lambda$ is asserted here.
\end{theorem}

\begin{proof}
Here $m=3$.  In the proof of Proposition~\ref{prop:a2}, all inequalities
coming from Lemma~\ref{lem:char2local}(i) must be equalities, because
\[
        2g-2=10=5m-5.
\]
Thus the branch locus consists of one point with full inertia $C_{12}$
and one tame branch point of index three.  In the local estimate
\[
        d=4t-1+2j_1+j_2\geq9t-1
\]
we have $t=3$, so equality gives
\[
        j_1=3,\qquad j_2=9.
\]
The upper numbering is compatible with quotients.  Herbrand's function
for the full inertia group gives the two positive upper jumps
\[
        u_1=\frac{3}{3}=1,\qquad
        u_2=1+\frac{9-3}{6}=2.
        \tag{5.19}
\]

Place the full wild branch at $z=0$ and the tame branch at $z=\infty$.
The quotient by the $C_4$-factor is a cyclic cubic cover of $\Pone$
totally ramified at these two points, hence after scaling the coordinate
it is
\[
        x^3=z.
\]
The quotient by the $C_3$-factor is a $C_4$-cover of the $z$-line,
branched only at $z=0$, with upper jumps $(1,2)$ by compatibility of
upper numbering with quotients.

Let $(f_0,f_1)$ be a reduced Witt representative for this $C_4$-cover.
The break formula of \cite[Theorem~2.3]{ElderKeating2026} gives
\[
        -v_0(f_0)=1,\qquad
        2=\max\{2,-v_0(f_1)\}.
\]
Because a reduced component has no pole order divisible by two, the
only possible polar terms are simple poles.  Since there are no other
branch points, the two components therefore have the form
\[
        f_0=\frac a z+A,\qquad
        f_1=\frac b z+B,
        \qquad a\neq0,
        \tag{5.20}
\]
with $A,B\in K$.

We remove the constants without changing the shape of the polar parts.
Indeed, add a constant coboundary whose first coordinate is $A$.  By
the Witt addition law \textup{(5.9)}, this kills the constant term in
$f_0$ and can only modify the coefficient of $1/z$ and the constant
term in $f_1$.  A second constant coboundary of the form $(0,C)$ then
removes the remaining constant in the second coordinate.  Thus the
cover is represented by
\[
        \left(\frac a z,\frac{b'}z\right)
\]
for some $b'\in K$.  Replace the old base coordinate by
$z_{\rm old}=a z$.  Then the Witt vector becomes
\[
        \left(\frac1z,\frac{\lambda}{z}\right),
        \qquad \lambda=\frac{b'}a.
\]
The cubic equation becomes $x^3=a z$; since $K$ is algebraically closed,
rescaling $x$ by a cube root of $a$ restores the form $x^3=z$.
This proves \textup{(5.18)}.  Notice that $\lambda=0$ is allowed; in that
case the second upper jump is still $u_2=2u_1=2$.

For completeness, the first component $1/z$ is not an
Artin--Schreier coboundary, so the Witt class in \textup{(5.18)} has
exact order four.  Its upper jumps are $(1,2)$, hence its lower jumps
are $(1,3)$ and its different exponent is $8$ at $z=0$.
Transitivity of the different in the compositum
gives
\[
        d_0=2+3\cdot8=26
\]
at the full $C_{12}$ branch point.  The tame branch at infinity has
index three and contributes four points, each with different exponent
two.  Hence
\[
        2g-2=-24+26+4\cdot2=10,
\]
so $g=6$.  Since the two subextensions have coprime degrees, the
compositum has Galois group
\[
        C_4\times C_3\cong C_{12}.
\]
\end{proof}

\section{The boundary theorem}

\begin{theorem}\label{thm:boundary}
Let $K$ be algebraically closed, let $X/K$ be a smooth projective curve
of genus $g\geq2$, and assume that
\[
        C_N\leq\Aut(X),
        \qquad N\in\{2g,2g-1\}.
\]
Then $X/C_N\cong\Pone$.

If $\Char(K)=0$, or if $\Char(K)=p>0$ with $p\nmid N$, the
possibilities are exactly those in Theorem~\ref{thm:tame}.  If
$\Char(K)=p\geq3$ and $p\mid N$, the Sylow $p$-subgroup has order $p$
and the possibilities are exactly those in Theorem~\ref{thm:oddwild}.
If $\Char(K)=2$, then for $N=2g-1$ the action is tame, while for
$N=2g$ the wild possibilities are precisely
\textup{(5.3)}, \textup{(5.6)}, \textup{(5.7)},
Theorem~\ref{thm:C4normal}, and Theorem~\ref{thm:C12normal}.
\end{theorem}

\begin{proof}
The quotient statement is Proposition~\ref{prop:rationalquotient}.  The
three characteristic cases are Theorems~\ref{thm:tame},
\ref{thm:oddwild}, \ref{thm:C4normal}, \ref{thm:C12normal} and
Propositions~\ref{prop:two-wild-char2}, \ref{prop:onewild-char2}.
The converse assertions were verified in the corresponding proofs.
\end{proof}

Table~\ref{tab:summary} collects all boundary cases.  For a wild branch
point we give the order of the inertia group and its positive lower
ramification breaks; for a tame branch point we give its ramification
index.

\begin{table}[ht]
\centering
\small
\renewcommand{\arraystretch}{1.15}
\begin{tabular}{@{}l c c l l@{}}
\toprule
$\Char(K)$ & $N$ & $g$ & ramification data & reference\\
\midrule
\multicolumn{5}{@{}l}{\emph{Tame actions: $\Char(K)=0$, or $\Char(K)=p\nmid N$}}\\
 & $N$ even, $N\geq4$ & $N/2$ & signature $(2,2,N,N)$ & Thm.~\ref{thm:tame}\\
 & $6$ & $3$ & signature $(2,3,3,6)$ & Thm.~\ref{thm:tame}\\
 & $12$ & $6$ & signatures $(2,3,4,12)$, $(3,3,4,4)$ & Thm.~\ref{thm:tame}\\
 & $30$ & $15$ & signature $(2,3,5,30)$ & Thm.~\ref{thm:tame}\\
 & $3$ & $2$ & signature $(3,3,3,3)$ & Thm.~\ref{thm:tame}\\
 & $15$ & $8$ & signature $(3,3,5,5)$ & Thm.~\ref{thm:tame}\\
\midrule
\multicolumn{5}{@{}l}{\emph{Wild actions, $p$ odd}}\\
$3$ & $3$ & $2$ & wild $C_3$ $[1]$, wild $C_3$ $[1]$ & \textup{(4.3)}\\
$3$ & $6$ & $3$ & wild $C_6$ $[4]$; tame $2$ & \textup{(4.8)}\\
$3$ & $12$ & $6$ & wild $C_3$ $[1]$; tame $4,4$ & \textup{(4.11)}\\
$3$ & $12$ & $6$ & wild $C_{12}$ $[4]$; tame $2,4$ & \textup{(4.16)}, \textup{(4.17)}\\
$5$ & $30$ & $15$ & wild $C_{30}$ $[6]$; tame $2,3$ & \textup{(4.16)}, \textup{(4.17)}\\
$3$ & $30$ & $15$ & wild $C_{30}$ $[10]$; tame $2,5$ & \textup{(4.16)}, \textup{(4.17)}\\
$3$ & $15$ & $8$ & wild $C_3$ $[1]$; tame $5,5$ & \textup{(4.11)}\\
$5$ & $15$ & $8$ & wild $C_5$ $[1]$; tame $3,3$ & \textup{(4.11)}\\
\midrule
\multicolumn{5}{@{}l}{\emph{Wild actions, $p=2$ (here $N=2g$; for $N=2g-1$ the action is tame)}}\\
$2$ & $2m$, $m\geq3$ odd & $m$ & wild $C_{2m}$ $[m]$, wild $C_{2m}$ $[m]$ & \textup{(5.3)}\\
$2$ & $6$ & $3$ & wild $C_6$ $[3]$; tame $3,3$ & \textup{(5.6)}\\
$2$ & $30$ & $15$ & wild $C_{30}$ $[15]$; tame $3,5$ & \textup{(5.7)}\\
$2$ & $4$ & $2$ & wild $C_4$ $[1,5]$ & Thm.~\ref{thm:C4normal}\\
$2$ & $12$ & $6$ & wild $C_{12}$ $[3,9]$; tame $3$ & Thm.~\ref{thm:C12normal}\\
\bottomrule
\end{tabular}
\caption{The cyclic groups $C_N\leq\Aut(X)$ with $N\in\{2g,2g-1\}$.
Brackets contain the positive lower ramification breaks.}
\label{tab:summary}
\end{table}

\begin{remark}
\rm{The list is consistent with lifting to characteristic zero.  By the
local--global principle (see \cite{Obus2012}) and the proof of the Oort
conjecture \cite{ObusWewers2014,Pop2014}, every pair $(X,C_N)$ in
characteristic $p>0$ lifts to characteristic zero, with the same genus
and the same group.  Consequently every pair $(N,g)$ occurring in the
wild part of Table~\ref{tab:summary} must already occur in its tame part,
which describes the situation over $\mathbb C$.  This is indeed the case.
Moreover, in each case the normalized different of every wild branch
point is the sum of the normalized differents of some branch points of a
complex signature with the same $N$ and $g$, as one expects from the
conservation of the different under lifting; for instance, for $p=3$ and $N=30$ one
has $\frac{49}{30}=\frac23+\frac{29}{30}$, matching the points of
indices $3$ and $30$ in $(2,3,5,30)$.}
\end{remark}

\begin{remark}
\rm{Combining Theorem~\ref{thm:boundary} with the characteristic-zero
classification of Irokawa and Sasaki \cite{IrokawaSasaki1995}, the
odd-characteristic classification of Dionigi--Giulietti--Timpanella
\cite{DGT2026}, and the characteristic-two classification of Timpanella
\cite{Timpanella2026}, one obtains the classification for cyclic
subgroups satisfying
\[
        N>2g-2.
\]
For the boundary cases treated here, every pair $(X,C_N)$ is represented
by one of the families above, up to the stated coordinate and generator
changes.  Distinct auxiliary parameter values may still define
isomorphic pairs unless uniqueness is explicitly asserted.}
\end{remark}

\section*{Acknowledgements}
The second author was partially supported by CNPq (grant no.~302774/2025-4),
FAEPEX (grant no.~3485/25), and FAPESP (grant no.~2024/00923-6).  

\section*{Statements and Declarations}

\medskip

\medskip
\noindent\textbf{Data availability.}
No datasets were generated or analysed during the current study.

\medskip
\noindent\textbf{Declaration on the use of artificial intelligence.}
Generative AI tools were used during the preparation of the manuscript to
assist with language editing, checking algebraic computations, and
stress-testing proof arguments.  All mathematical statements, proofs, and
references were independently verified by the authors, who take full
responsibility for the content of the manuscript.

\end{document}